\documentclass[a4paper,twoside,12pt]{article}
\usepackage[english]{babel}

\usepackage{amsmath,amsfonts,amssymb,amsthm}
\newtheorem{teo}{Theorem}

\newtheorem{rem}{Remark}

 \title{{\bf   On one application of the Zigmund-Marczinkevich theorem}}

\author{Maksim \,V.~Kukushkin   \\ \\
 \small  \textit{Moscow State University of Civil Engineering, 129337,  Moscow, Russia}\\
 \small\textit{Kabardino-Balkarian Scientific Center, RAS, 360051,  Nalchik, Russia}\\
\textit{\small\textit{kukushkinmv@rambler.ru}} }

\date{}
\begin{document}

\maketitle

\begin{abstract}
In this paper we aim to  generalize    results obtained    in the framework of fractional calculus by the way of reformulating them   in terms  of operator theory. In its own turn,   the achieved generalization allows us to spread the obtained technique on practical problems  that connected with various physical -
chemical processes.

\end{abstract}
\begin{small}\textbf{Keywords:} Positive operator; fractional power of an operator; semigroup generator;
 strictly accretive property.
\end{small}

\section{Introduction}

To write this paper, we were firstly motivated by  the boundary value  problems connected with
  various physical -
chemical processes: filtration of liquid and gas in highly porous fractal   medium; heat exchange processes in medium  with fractal structure and memory; casual walks of a point particle that starts moving from the origin
by self-similar fractal set; oscillator motion under the action of
elastic forces which is  characteristic for  viscoelastic media, etc. It is worth noticing that the attention of many engineers are being  attracted to the following processes:     dielectric polarizations   \cite{firstab_lit: Sun H.H.}, electrochemical processes \cite{firstab_lit:Ichise M.},\cite{firstab_lit:H. Rekhviashvili},\cite{firstab_lit: 1Sun H.H.},  colored noises
\cite{firstab_lit: Mandelbrot B.},  chaos \cite{firstab_lit: Hartley T.T}.
 The  special   interest   is devoted to  the viscoelastic materials
 \cite{firstab_lit: Bagley R.L.},\cite{firstab_lit: Koeler R.C.},\cite{firstab_lit: 1Koeler R.C.},\cite{firstab_lit: Skaar S.B.}.
 As it is well-known, to describe these processes we have to involve the
theory of differential equations of fractional order and such notions  as the Riemann-Liouville, Marchaud, Weyl, Kipriyanov fractional derivatives, which     have applications to the physical -
chemical processes listed above.  For instance,
the model  of the so-called  intensification of scattering processes  is based on the   fractional telegraph equation which was studied in the papers \cite{firstab_lit:1Mamchuev},\cite{firstab_lit:2Mamchuev}. In its own turn,
	the foundation  of   models  describing the  processes listed above can be obtained by virtue of fractional calculus methods, the central point of which   is a concept of the Riemann-Liouville operator acting in the weighted Lebesgue space.
In particular,  we would like   to point out the reader  attention to   the boundary value  problems  for the second order   differential operator  with the Riemann-Liouville fractional derivative in final terms. Many papers were  devoted to this question, for instance
\cite{firstab_lit:1Aleroev1984},\cite{firstab_lit:3Aleroev2000},\cite{kukushkin2019},\cite{firstab_lit:1Nakhushev1977}. It is quite reasonable that the  variety of fractional derivative senses mentioned above create a motivation to consider   abstract methods in order to solve concrete problems without inventing an additional technique in each case.
   In its own turn, the operator theory methods play an important role in  applications and need not any of special advertising. Having forced by these reasons,  we deal with mapping theorems for  operators acting on Banach spaces in order to obtain afterwards the desired results applicable to   integral operators.
   We also note that our interest was inspired by lots of previously known  results related to mapping theorems for fractional integral  operators      obtained by mathematicians  such as
    Rubin B.S. \cite{firstab_lit: Rubin},\cite{firstab_lit: Rubin 1},\cite{firstab_lit: Rubin 2}, Vakulov B.G. \cite{firstab_lit: Vaculov},   Samko S.G.  \cite{firstab_lit: Samko M. Murdaev},\cite{firstab_lit: Samko Vakulov B. G.}, Karapetyants N.K.
\cite{firstab_lit: Karapetyants N. K. Rubin B. S. 1},\cite{firstab_lit: Karapetyants N. K. Rubin B. S. 2}.

\section{Preliminaries}

We consider a pair of  spaces with a finite  measure $(\Omega ,\mathcal{F},\mu_{i}),\,(i=0,1),$ the  corresponding Banach spaces of real functions defined on a set  $\Omega$ are denoted by $ L_{p}(\Omega,\mu_{i}),\,1< p < \infty.$   As usual, we assume that $L_{p}(\Omega,\mu_{i}) $ are reflexive and the Riesz representation theorem is true. The dual normed spaces are denoted by $L_{p'}(\Omega,\mu_{i})$ respectively, where $1/p+1/p'=1.$ Suppose $L_{p}(\Omega,\mu_{0})\cap L_{p}(\Omega,\mu_{1})\neq\emptyset,$   there exists a basis $\{e_{n}\}_{0}^{\infty}$ in  $L_{p}(\Omega,\mu_{1}),\,\{e_{n}\}_{0}^{\infty}\subset L_{p}(\Omega,\mu_{0})\cap L_{p}(\Omega,\mu_{1}),$ which is also a basis in $  L_{p'}(\Omega,\mu_{1}).$ Under this assumption, it is quite reasonable to involve the auxiliary consruction
 $$
 S_{k}f:=\sum\limits_{n=0}^{k}f_{n}e_{n},\,f_{n}:=\int\limits_{\Omega}fe_{n}d\mu_{1},\,k=0,1,...\,,\,f\in L_{p}(\Omega,\mu_{1}).$$
   We  need the following essential assumption, the measure $\mu_{1}$   and the domain $\Omega$ are such that the Zigmund-Marczinkevich theorem  (see \cite{firstab_lit: Marz}, Appendix) is applicable.
Let $A,B$ be a pair of linear operators boundary acting on $L_{p}(\Omega,\mu_{0}),L_{p}(\Omega,\mu_{1})$ respectively, moreover both operators have the same restriction on $L_{p}(\Omega,\mu_{0})\cap L_{p}(\Omega,\mu_{1}).$   Assume that the   operator  $A^{-1} $ is defined on $\{e_{n}\}$ and
 denote the following functionals by
 \begin{equation}\label{1}
 (e_{m},A e_{n})_{ \mu_{1}}  = A_{mn},\,(e_{m},A^{-1}e_{n})_{  \mu_{1}} = A'_{mn},
 \end{equation}
here and further we use the denotations
$$
(f,g )_{ \mu_{i}}:= \int\limits_{\Omega}fg \,d\mu_{i},\;(f,g )_{ \mu_{0},n}:= \int\limits_{\Omega_{n}}fg \,d\mu_{0}.
 $$
\section{Main results}

The following theorems are formulated in terms of    the  Zigmund-Marczinkevich theorem (see Appendix) and provide a description of mapping properties of  the quite   wide operators  class   including fractional integral operators.
\begin{teo}\label{T1}
Suppose $2\leq p< \infty,\,\beta,\lambda \in [0,\infty),$
$$
 \psi\in L_{p}(\Omega,\mu_{1}) ,\, \left|\sum\limits_{n=0}^{\infty} \psi_{n}A_{m n}\right| \sim m^{-\lambda},  \,M_{m}\sim m^{\beta},\,m\rightarrow \infty ;
$$
then
$$A \psi =f\in L_{q}(\Omega,\mu_{1}), $$
 where $q=p,$ if  $\,0\leq\lambda\leq1/2;$
$q$ is an arbitrary large number     satisfying
\begin{equation}\label{3}
q<\frac{\nu(2\beta+1)}{\nu(\beta+1-\lambda)+2\lambda-1},
\end{equation}
if $1/2<\lambda<[\nu(\beta+1)-1]/(\nu-2);$ and $q$ is an arbitrary large number, if $  \lambda\geq [\nu(\beta+1)-1]/(\nu-2).$
 Moreover the image     is represented by a convergent in $L_{q}(\Omega,\mu_{1})$   series
\begin{equation*}
f(x) =\sum\limits_{m=0}^{\infty}e_{m}(x)  \sum\limits_{n=0}^{\infty} \psi_{n}A _{mn}.
\end{equation*}
This theorem can be formulated in the  matrix form
$$
  A \times\psi=f\,,\;\;\;\;\sim\;\;\;\;
  \begin{pmatrix}A _{00}& A _{01}&...\\
 A _{10}& A _{11}&...\\
\cdot\\
\cdot\\ \cdot&&...
\end{pmatrix}\times\begin{pmatrix}\psi_{0}\\\psi_{1}\\ \cdot\\ \cdot \\ \cdot \end{pmatrix}= \begin{pmatrix}f_{0}\\f_{1}\\ \cdot\\ \cdot \\ \cdot \end{pmatrix}.
$$
 \end{teo}
\begin{proof}
Note,  that in accordance with the basis property of $\{e_{n}\},$ we have
$$
 \sum\limits_{n=0}^{l}\psi_{n}e_{n}\stackrel{L_{p}(\Omega,\mu_{1})}{\longrightarrow} \psi\in L_{p}(\Omega,\mu_{1}),\;l\rightarrow \infty.
$$
Since the operator $A$ is bounded, then
$$
\sum\limits_{n=0}^{l}\psi_{n}A e_{n}  \stackrel{L_{p}(\Omega,\mu_{1})}{\longrightarrow} A\left(\sum\limits_{n=0}^{\infty}\psi_{n}e_{n}\right)=A\psi,\;l\rightarrow \infty.
$$
Hence
$$
\sum\limits_{n=0}^{l}\psi_{n}\left(A e_{n},e_{m}\right)_{ \mu_{1} } \longrightarrow\left(  A \psi,e_{m}\right)_{ \mu_{1} },\;l\rightarrow \infty .
$$
Applying   first     formula  \eqref{1}, we obtain
$$
f_{m}=\left( A\psi,e_{m}\right)_{ \mu_{1} }
=  \sum\limits_{n=0}^{\infty} \psi_{n}A_{m n}.
$$
 Since $0\leq\lambda\leq1/2,$ then it is not hard to prove that
\begin{equation*}
\frac{\nu(2\beta+1)}{\nu(\beta+1-\lambda)+2\lambda-1}\leq2.
\end{equation*}
This implies that we cannot use condition \eqref{3} to obtain the additional information on the function $f.$ However,  note that in any case  by virtue of the conditions imposed on the operator $A,$ we have $p=q.$
Using condition \eqref{3}, by simple calculation  in the case $1/2<\lambda<[\nu(\beta+1)-1]/(\nu-2),$ we obtain
\begin{equation}\label{4}
\beta\frac{\nu(q-2)}{\nu-2}+\frac{\nu-1}{\nu-2}(q-2)-\lambda q<-1.
\end{equation}
It implies that corresponding series \eqref{2} is convergent. By virtue of this fact, having applied  the Zigmund-Marczinkevich theorem, we obtain the desired result.
 Now assume that $[\nu(\beta+1)-1]/(\nu-2)\leq\lambda,$ then consider relation \eqref{4} and let us prove that it is fulfilled for $2\leq q< \infty.$ By easy calculations, we can rewrite \eqref{4} in the following form
\begin{equation}\label{5}
q\left\{ \frac{\beta\nu}{\nu-2}+\frac{\nu-1}{\nu-2}-\lambda \right\}<
2\left\{\frac{\beta\nu}{\nu-2}+\frac{\nu-1}{\nu-2}\right\}-1.
\end{equation}
Since the multiplier of $q$ is non-positive and the right side of   inequality \eqref{5} is positive (the proof of this fact is left to the reader), then we have the fulfilment of inequality \eqref{5} for $2\leq q< \infty.$ Hence \eqref{4} holds and  using the same reasonings we obtain the desired result.
\end{proof}
The following result is  formulated in terms of the  coefficients of the expansion  on the basis $\{e_{n}\}$  and     devoted to the conditions under which  the inverse operator $B^{-1}$ exists.
Consider the following  operator   equation under  most general assumptions relative to the right part
\begin{equation}\label{6}
B\varphi =f .
\end{equation}
Let us consider the  Riemann-Liouville operator to demonstrate applicability of such a consideration. In this case relation \eqref{6} provides the generalized   Abel equation (see \cite{kukushkin2019axi}) that becomes the ordinary Abel equation, if we make the following assumptions concerning to the right part.
 Thus, we know  if the next conditions hold   $I^{1-\alpha}_{a+}f\in AC (\bar{I}),\; (I^{1-\alpha}_{a+}f)    (a)=0,\,I=(a,b)\subset \mathbb{R},$ then there  exists a unique solution of   the ordinary Abel equation    in the class $L_{1}(I)$   (see\cite{firstab_lit:samko1987}).
The   sufficient conditions for   solvability of    equation \eqref{6} in the abstract case   are established in the  following theorem.
\begin{teo}
Suppose $\lambda,\beta\in [0,\infty),\,   M_{n}\sim n^{\beta},$  the right part of    equation \eqref{6} such that
\begin{equation}\label{7}
 \left\|A^{-1}\!S_{k} f \right\|_{L_{p}(\Omega,\mu_{1})}\! \leq C,\;k\in \mathrm{N}_{0},\;\;\left|\sum\limits_{n=0}^{\infty} f_{n}A'_{mn}\right|\sim m^{-\lambda},\;m\rightarrow \infty ;
\end{equation}
then  there exists  a     solution of    equation \eqref{6} in     $L_{p}(\Omega,\mu_{1}),$  the   solution belongs to   $L_{q}(\Omega,\mu_{1}),$
 where $q=p,$ if  $\,0\leq\lambda\leq1/2;$
$q$ is an arbitrary large number     satisfying
\begin{equation}\label{8}
q<\frac{\nu(2\beta+1)}{\nu(\beta+1-\lambda)+2\lambda-1},
\end{equation}
if $1/2<\lambda<[\nu(\beta+1)-1]/(\nu-2);$ and $q$ is an arbitrary large number, if $  \lambda\geq [\nu(\beta+1)-1]/(\nu-2).$
 Moreover the solution   is represented by a convergent in $L_{q}(\Omega,\mu_{1})$   series
\begin{equation}\label{9}
\psi(x)=\sum\limits_{m=0}^{\infty}e_{m}(x) \sum\limits_{n=0}^{\infty} f_{n}A'_{mn}.
\end{equation}
Moreover, if  the measures $\mu_{i}$ such that: $\mu_{0}(M)=0\,\Rightarrow\, \mu_{1}(M)=0,\,\forall M\subset\Omega,$ there exists a sequence of sets
$$
 \bigcup\limits_{n=1}^{\infty}\Omega_{n}=\Omega,\;\Omega_{n}\subset \Omega_{n+1}, \; L_{p}(\Omega,\mu_{1})\subset L_{p}(\Omega_{n},\mu_{0}),\,\mu_{0}(\Omega\!\setminus\!\Omega_{n})\rightarrow 0,\,n\rightarrow \infty,
$$
the corresponding sets of functions $\Theta_{n},\,n\in \mathbb{N}_{0}$  such  that
$$
\Theta_{1} \subset \Theta_{2} \subset...\subset\Theta ,
$$
and $\Theta_{n},\Theta$ are   dense  sets in the spaces $L_{p'}(\Omega_{n},\mu_{0}),L_{p'}(\Omega,\mu_{0})$ respectively,
$$
 (g,\eta)_{\mu_{0},n}=(g,\eta)_{\mu_{0}},\,\eta\in \Theta_{n},\, g\in L_{p}(\Omega_{n},\mu_{0}),
 $$
 $$
 \forall\eta \in \Theta,\,\forall\xi\in L_{p}(\Omega,\mu_{1}) ,\,\exists g\in L_{p'}(\Omega,\mu_{1}):\, (\xi,\eta)_{\mu_{0}}=(\xi,B^{\ast}g)_{\mu_{1}},
$$
then the existing  solution is unique.
\end{teo}
\begin{proof}
  Applying    formula \eqref{1} and using the theorem conditions,   we obtain  the following relation
\begin{equation}\label{10}
\left( A^{-1}S_{k} f ,e_{m}  \right)  \longrightarrow\sum\limits_{n=0}^{\infty}  f_{n}A'_{mn},\;k\rightarrow \infty,\,m\in  \mathbb{N}_{0}.
\end{equation}
 Since   relation \eqref{10} holds and  the sequence $\left\{A^{-1}S_{k} f\right\}_{0}^{\infty}$   is bounded with respect to the norm    $L_{p}(\Omega,\mu_{1}),$ then due to the  well-known theorem,  we have that the sequence  $\left\{A^{-1}S_{k} f\right\}_{0}^{\infty}$ weakly   converges  to  some function
 $\psi\in L_{p}(\Omega,\mu_{1}).$    Using the ordinary properties of inverse and adjoint operators, taking into account that $A^{-1}S_{k} f\in L_{p}(\Omega,\mu_{1}),$   we obtain the representation
 $$
\left( S_{k} f ,e_{m} \right)_{\mu_{1}} = \left(A A^{-1} S_{k} f ,e_{m} \right)_{\mu_{1}}= \left(B_{0} A^{-1} S_{k} f ,e_{m} \right)_{\mu_{1}}=
$$
$$
  =\left( A^{-1}S_{k} f , B^{\ast}_{0} e_{m}  \right)_{\mu_{1}}=\left( A^{-1}S_{k} f , B^{\ast}  e_{m}  \right)_{\mu_{1}},
 $$
 where $B_{0}$ is a restriction of $B$ to the set $A^{-1}S_{k} f,\,k\in \mathbb{N}_{0},$ thus we have $B^{\ast}_{0}\supset B^{\ast}. $
 Due to the week convergent of the sequence $\left\{A^{-1}S_{k} f\right\}_{0}^{\infty},$ we have
 $$
\left( A^{-1}S_{k} f ,B^{\ast} e_{m} \right)_{\mu_{1}}\rightarrow \left( \psi,B^{\ast} e_{m} \right)_{\mu_{1}}=\left( B\psi,  e_{m} \right)_{\mu_{1}},\,k\rightarrow\infty.
$$
It follows that
 \begin{equation}\label{11}
 \left( S_{k} f ,e_{m} \right)_{\mu_{1}} \rightarrow \left( B\psi,  e_{m} \right)_{\mu_{1}},\,k\rightarrow\infty.
 \end{equation}
 Taking into account that
\begin{equation*}
\left( S_{k} f ,e_{m} \right)_{\mu_{1}}=  \left\{ \begin{aligned}
 f_{m},\;k\geq m,\\
  \!0  ,\;  k<m\, \\
\end{aligned}
 \right.\;\;,
\end{equation*}
we obtain
$$
\left(  B\psi ,e_{m}    \right)_{\mu_{1}} =f_{m},\,m\in  \mathrm{N}_{0}.
$$
 Using the  uniqueness property of   basis expansion, we obtain that  the equality  $B\psi=f$ holds  on the set $\Omega \setminus M,\,\mu_{1}(M)=0.$
Hence  there  exists a  solution  of    equation \eqref{6}.  Now let us proceed to the following part of the  proof.
Note  that  we have  previously  proved  the fact   $\psi \in L_{p}(\Omega,\mu_{1}), $   if $0\leq\lambda<\infty.$
Let us show that   $\psi \in L_{q}(\Omega,\mu_{1}), $ where $q$ is defined by condition \eqref{8}, if $1/2<\lambda<[\nu(\beta+1)-1]/(\nu-2) .$   In accordance with the above reasonings,    we have
$$
\left(A^{-1}S_{k} f ,e_{m} \right)_{\mu_{1}}  \longrightarrow \left( \psi,e_{m} \right)_{\mu_{1}},\,k\rightarrow\infty ,\;m\in  \mathbb{N}_{0}.
$$
Combining this fact with   \eqref{10}, we get
\begin{equation}\label{12}
\psi_{m}=\left( \psi,e_{m} \right)_{\mu_{1}}   =\sum\limits_{n=0}^{\infty} f_{n}A'_{mn},\;m\in   \mathbb{N}_{0}.
\end{equation}
Using the  theorem conditions,  we have
$$
|\psi_{m}|\sim m^{-\lambda},\,m\rightarrow \infty.
$$
Now we need to apply the  Zigmund-Marczinkevich theorem. For this purpose, let us show that  under the assumption $1/2<\lambda<[\nu(\beta+1)-1]/(\nu-2) $ the relation \eqref{4} is fulfilled  in terms of this theorem conditions.
 Having done the reasonings analogous to the reasonings of Theorem \ref{T1}, we obtain the desired result. Thus we have  $\psi \in L_{q}(\Omega,\mu_{1}), $ where $  q$ is defined by condition \eqref{8}. The proof corresponding to the case
$  \lambda\geq [\nu(\beta+1)-1]/(\nu-2)$ is also analogous to the proof given in Theorem \ref{T1}.    In a similar way, we  obtain in this case  that $\psi \in L_{q}(\Omega,\mu_{1}), $ where   $q$ is arbitrary large.  Taking into account the facts  given above, by virtue of   the  Zigmund-Marczinkevich theorem, we also have the fulfilment of relation \eqref{9}, if $\lambda>1/2.$
Let us prove that under the second part of the  theorem assumptions the existing solution is unique.     Assume  that there  exists   another solution $\phi\in L_{p}(\Omega,\mu_{1}),$  and let us   denote $\xi:=\psi-\phi.$
 Due to the theorem conditions, we have
$$
\left(\xi,\eta \right)_{\mu_{0}} =\left(\xi,B^{\ast} g \right)_{\mu_{1}} =\left(B\xi,  g \right)_{\mu_{1}} =0,\,\eta\in  \Theta,\,g\in  L_{p'}(\Omega,\mu_{1}).
$$
Hence
 $$
\left(\xi,\eta \right)_{\mu_{0},n} =0,\,\forall \eta \in \Theta_{n}.
$$
We claim that  $\xi\neq 0.$ Therefore   in accordance  with the consequence of the Hahn-Banach theorem there exists the element $\vartheta\in L_{p'}(\Omega_{n},\mu_{0}),$ such that
$$
 \left(\xi,\vartheta \right)_{\mu_{0},n}  =\|\psi-\phi\|_{L_{p }(\Omega_{n},\mu_{0})}>0.
$$
On the other hand, there exists  the sequence $\{\eta_{k}\}_{1}^{\infty}\subset \Theta_{n},$ such  that $\eta_{n} \rightarrow \vartheta$ with respect to the norm $L_{p'}(\Omega_{n},\mu_{0}).$  Hence
$$
0=\left(\xi,\eta_{n} \right)_{\mu_{0},n} \rightarrow \left(\xi,\vartheta \right)_{\mu_{0},n}.
$$
Thus, we have come to the contradiction. Due to this fact, it is not hard to prove that $\psi=\phi$ on the set $\Omega\setminus M,\,\mu_{0}(M)=0.$    It implies that $\psi=\phi$ on the set $\Omega\setminus M,\,\mu_{1}(M)=0.$      The    uniqueness has been proved.
\end{proof}

\section{Applications}

  In this section our aim is to justify the application   of the obtained  abstract results  to the   fractional integral operator. Throughout this section   we  use the following denotation for the  weighted complex  Lebesgue spaces   $L_{p}(I, \mu),\,1\leq p<\infty,$  where $I=(a,b)$ is an interval  of the real axis and $d\mu=\omega(x) dx,\,x\in I,$ where  the  weighted function $\omega(x)=(x-a)^{\beta}(b-x)^{\gamma},\,\beta,\gamma \geq-1/2 .$       If $\omega=1,$ then we use the     notation  $L_{p}(I).$
Using the denotations of the paper \cite{firstab_lit:samko1987}, let us  define  the left-side  fractional integral  and   derivative   of   real order $\alpha\in(0,1)$ respectively
 $$
\left(I_{a+}^{\alpha}f\right)(x)=\frac{1}{\Gamma(\alpha)} \int\limits_{a}^{x}\frac{f(t)}{(x-t)^{1-\alpha}}\,dt  ,\;f\in L_{1}(I),\,
 \left(D^{\alpha}_{a+}f\right)(x)=\frac{d }{dx } \left(I_{a+}^{1-\alpha}f\right)(x),\,f\in I_{a+}^{\alpha}(L_{1}),
$$
where $I_{a+}^{\alpha}(L_{1}) $ is  the class  of functions that can be represented by the fractional integral  defined on $L_{1}(I)$ (see\cite{firstab_lit:samko1987}).
The orthonormal system of the  Jacobi  polynomials, with the    parameters $\beta,\gamma$ corresponding to the weighted function,  is denoted by
$$
p_{n}^{\,(\beta,\gamma)}(x)= \delta_{n} (x-a)^{-\beta}(b-x)^{-\gamma}\frac{d^{n}}{dx^{n}}\left[(x-a)^{\beta+n}(b-x)^{\gamma+n}\right],\,n\in \mathbb{N}_{0},
$$
where $\delta_{n}$ are constants depending of $\beta,\gamma.$
For   the case   corresponding to $ \beta=\gamma=0,$ we have the  Legendre  polynomials.
It is worth seing that   the criterion of the basis property for the Jacobi polynomials    was  proved by Pollard H.  in the work \cite{firstab_lit:H. Pollard 3}. In that paper   Pollard H. formulated  the theorem proposing   that the Jacobi polynomials have the basic property in the space $L_{p}(I_{0},\mu),\,I_{0}:=(-1,1),\;\beta,\gamma\geq -1/2,\,M(\beta,\gamma)<p<m(\beta,\gamma)$  and do not have the basis property, if $ p<M(\beta,\gamma)$ or $ p>m(\beta,\gamma),$ where $M(\beta,\gamma),\,m(\beta,\gamma)$ are constants depending on $\beta,\gamma.$
According to the denotations of the paper \cite{kukushkin2019axi}, we have
$$
A^{\alpha,\beta,\gamma}_{mn}:= \hat{\delta}_{m}\sum\limits_{k=0}^{n}  (-1)^{ k}  \frac{  \mathfrak{C}_{n}^{(k)}(\beta,\gamma)B(\alpha+\beta+k+1,\gamma+m+1) }{\Gamma(k+\alpha-m+1)},
$$
where $\hat{\delta}_{m}$ is a constant depending on $\alpha,\beta,\gamma.$  Suppose $\beta,\gamma\in[-1/2,1/2],\,M(\beta,\gamma)<p<m(\beta,\gamma).$ Further, we will use the short-hand notation $p_{n}:=p^{(\beta,\gamma)}_{n}.$
In accordance with the results of the paper \cite{kukushkin2019axi}, we have
\begin{equation}
\int\limits_{I}p _{m} I_{a+}^{\alpha}p _{n} d\mu
  =(-1)^{n}A^{\alpha,\beta,\gamma}_{mn},\;\;
 \int\limits_{I}p _{m} D_{a+}^{\alpha}p _{n} d\mu
 =(-1)^{n}A^{-\alpha,\beta,\gamma}_{mn}.
 \end{equation}
Let us assume that
$$
 \Omega:=I,\, \mu_{0}:=x,\,\mu_{1}:=\mu,\,L_{p}(\Omega,\mu_{0}):=L_{p}(I),\,L_{p}(\Omega,\mu_{1}):=L_{p}(I,\mu),\,e_{n}:=p_{n},
$$
$$
A:= I_{a+}^{\alpha},\,A^{-1}:=D_{a+}^{\alpha},\,A_{mn}:=A^{\alpha,\beta,\gamma}_{mn},\,A'_{mn}:=A^{-\alpha,\beta,\gamma}_{mn},\,
 B:=I_{a+}^{\alpha}.
$$
Taking into account the facts given above, it can be easily  proved (see \cite{kukushkin2019axi}) that all assumptions  of section 1 are fulfilled   except   the following ones.   In terms of the given interpretation,    the proof of uniqueness of the Abel equation solution does not seem obvious. Thus we need to produce it,
for this purpose  assume that
$$
\Omega_{n}:=\left(a+ \frac{1}{n},b- \frac{1}{n}\right),\,  \Theta_{n}:=C_{0}^{\infty}(\Omega_{n}),\,\Theta:=C_{0}^{\infty}(\Omega).
$$
Then the following assumptions of the uniqueness part of Theorem \ref{2}  are fulfilled
$$
 \bigcup\limits_{n=1}^{\infty}\Omega_{n}=\Omega,\;\Omega_{n}\subset \Omega_{n+1}, \; L_{p}(\Omega,\mu_{1})\subset L_{p}(\Omega_{n},\mu_{0}),\,\mu_{0}(\Omega\!\setminus\!\Omega_{n})\rightarrow 0,\,n\rightarrow \infty.
$$
 The verification is left to the reader. It is also not hard to prove that
$$
 (g,\eta)_{\mu_{0},n}=(g,\eta)_{\mu_{0}},\,\eta\in \Theta_{n},\, g\in L_{p}(\Omega_{n},\mu_{0}).
 $$
 Let us show that
 $$
 \forall\eta \in \Theta,\,\forall\xi\in L_{p}(\Omega,\mu_{1}) ,\,\exists h\in L_{p'}(\Omega,\mu_{1}):\, (\xi,\eta)_{\mu_{0}}=(\xi,B^{\ast}h)_{\mu_{1}}.
$$
Consider
$$
    \omega^{-1}(x)D^{\alpha}_{b-}\eta(x)= (x-a)^{-\beta}(b-x)^{1-\gamma }(b-x)^{-1} D^{\alpha}_{b-}\eta(x).
$$
If we note that $D^{\alpha}_{b-}\eta(b)=0$ (see \cite{firstab_lit:samko1987}) then it can be easily shown that
$$
(b-x)^{-1} D^{\alpha}_{b-}\eta(x)\rightarrow D^{\alpha+1}_{b-}\eta(b),\,x\rightarrow b.
$$
Hence the function  $
\omega^{-1} D^{\alpha}_{b-}\eta
$
is bounded. It implies that we have a representation $D^{\alpha}_{b-}\eta=\omega h,$ where $h$ is a bounded function and due to this fact it belongs to $L_{p'}(\Omega,\mu_{1}).$ By virtue of the fact $\eta\in C_{0}^{\infty}(\Omega)$(see \cite{firstab_lit:samko1987}),   we have $I^{\alpha}_{b-}D^{\alpha}_{b-}\eta=\eta.$ Hence
$$
\eta=I^{\alpha}_{b-}\omega h,\,h\in L_{p'}(\Omega,\mu_{1}).
$$
Taking into account the reasonings given above, we get
$$
\int\limits_{\Omega}\xi \eta\, d\mu_{0} =\int\limits_{\Omega}\xi \eta\, d\mu_{0}=\int\limits_{\Omega}\xi\, \omega^{-1}I^{\alpha}_{b-}\omega h\, d\mu_{1}.
$$
On the other hand, since we have
$$
|l_{\eta}(\xi)|=\left|\int\limits_{\Omega}\xi \eta\, d\mu_{0}\right|=\left|\int\limits_{\Omega}\xi \eta\,\omega^{-1} d\mu_{1}\right|\leq \|\xi\|_{L_{p}(\Omega,\mu_{1})}
\|\eta\,\omega^{-1}\|_{L_{p'}(\Omega,\mu_{1})},
$$
then using the  Riesz representation theorem, we get
$$
\int\limits_{\Omega}\xi \eta\, d\mu_{0}=\int\limits_{\Omega}\xi \eta^{\ast}\, d\mu_{1},\,\eta^{\ast}\in L_{p'}(\Omega,\mu_{1}).
$$
Hence $\omega^{-1}I^{\alpha}_{b-}\omega h\in L_{p'}(\Omega,\mu_{1}) $ and we obtain the desired result, where $B^{\ast}h:= \omega^{-1}I^{\alpha}_{b-}\omega h.$ Applying the analogous reasoning we prove that $B^{\ast}e_{m}=\omega^{-1}I_{b-}^{\alpha}\omega p_{m},\,(m=0,1,...,).$ As a result we have come to conclusion the theoretical results of section 1 have a concrete application to the questions connected with existence and uniqueness of the Abel equation  solution  in the weighted  case, which has an important role in the description of    physical   processes in porous medium.

\section{Conclusions}

In this paper, our first  aim is to  construct an operator model
 describing the fractional integral action in the weighted  Lebesgue spaces.  The approach used in the paper is the following: to generalize   known results of fractional calculus and by this way achieve a novel method of studying operators action  in Banach spaces.
 Besides the theoretical results of the paper, we produce the relevance of such a consideration that  provided by a plenty of applications  in  various engendering sciences. More precisely a description of such process as electrochemical processes,   dielectric polarizations,   colored noises  are provided by the theoretical part of this paper.
The foundation  of  models  describing the  processes listed above can be obtained by virtue of fractional calculus methods, the central point of which is a concept of the Riemann-Liouville operator acting in the weighted Lebesgue space.
In its own turn this concept is covered by the theoretical part of this paper.
Thus, the obtained results are  harmoniously connected with the concrete models of physical-chemical process.

\section*{Appendix}
In this section, we formulate  the Zigmund-Marczinkevich theorem (see \cite{firstab_lit: Marz}).

\noindent{\rm Theorem.}
{\it
Suppose that
$$
\left(\int\limits_{\Omega}|e_{n}|^{\nu}d\mu_{1}\right)^{1/\nu}\leq M_{n},\,2\leq q<\nu,\,M_{n}\leq M_{n+1},\, (n=0,1,...),
$$
and the series
$$
\sum\limits_{n=0}^{\infty}|c_{n}|^{q}M_{n}^{\frac{\nu(q-2)}{\nu-2}}n^{\frac{\nu-1}{\nu-2}(q-2)}<\infty
$$
converges. Then there exists a function $f\in L_{q}(\Omega,\mu_{1})$ such that
\begin{equation}\label{2}
\left(\int\limits_{\Omega}|f|^{q}d\mu_{1}\right)^{1/q}\leq A_{q,\nu}\left( \sum\limits_{n=0}^{\infty}|c_{n}|^{q}M_{n}^{\frac{\nu(q-2)}{\nu-2}}n^{\frac{\nu-1}{\nu-2}(q-2)} \right)^{1/q}.
\end{equation}
Here $A_{q,\nu}$ depends on $q$ and $\nu$ only. If by $A_{q,\nu}$ we mean the least number satisfying   \eqref{2}  for all sequences $\{c_{n}\},$ then
$$
A_{q,\nu}\leq G\frac{\nu-2}{\nu-q}q,
$$
where $G$ is an absolute constant.
}

\end{document}